\documentclass[12pt,reqno]{amsart}
\usepackage{amssymb}
\usepackage{mathtools}
\usepackage[english]{babel}
\usepackage{epsfig}
\usepackage{mathptmx}
\usepackage{xcolor}
\usepackage{amsmath,amsfonts,amssymb}
\usepackage{mathtools}
\usepackage{mathrsfs}
\usepackage[all]{xy}
\usepackage{graphicx}
\usepackage{latexsym}
\usepackage{verbatim}

\numberwithin{equation}{section}

\def\1#1{\overline{#1}}
\def\2#1{\widetilde{#1}}
\def\3#1{\widehat{#1}}
\def\4#1{\mathbb{#1}}
\def\5#1{\frak{#1}}
\def\6#1{{\mathcal{#1}}}

\newtheorem{theorem}{Theorem}[section]

\newtheorem{proposition}[theorem]{Proposition}
\newtheorem{corollary}[theorem]{Corollary}

\theoremstyle{definition}

\theoremstyle{remark}
\newtheorem{remark}[theorem]{Remark}

\numberwithin{equation}{section}

\title{On left inverses in the Lempert theorem}
\author{W\l odzimierz Zwonek}

\address{Department of Mathematics, Faculty of Mathematics and Computer Science, Jagiellonian University, \L ojasiewicza 6, 30-348 Krak\'ow, Poland}
\email{wlodzimierz.zwonek@uj.edu.pl}
\thanks{Supported by the NCN (National Science Centre, Poland) grant no. 2023/51/B/ST1/01312}
\date{November 2025}

\begin{document}

\begin{abstract} In the paper we discuss the problem of existence, uniqueness and extension through the boundary of left inverses to complex geodesics in Lempert domains. We concentrate on special left inverses (so called Lempert left inverses) characterized by the fact that their fibers are intersections of affine hyperplanes with the domain.\end{abstract}

\maketitle

\section{Introduction} 
For a domain $D\subset\mathbb C^n$ we define {\it the Carath\'eodory pseudodistance}
\begin{equation}
    c_D(w,z):=\sup\left\{p(F(w),F(z)):D\in\mathcal O(D,\mathbb D)\right\}
\end{equation}
and {\it the Lempert function}
\begin{equation}
l_D(w,z):=\inf\left\{p(\lambda_1,\lambda_2):\exists f\in\mathcal O(\mathbb D,D), f(\lambda_1)=w,f(\lambda_2)=z\right\},\; w,z\in D,
\end{equation}
where $p$ is the Poincar\'e distance on the unit disc. We also put $c_D^*:=\tanh c_D$ and $l_D^*:=\tanh l_D$. A fundamental Lempert theorem states the following (\cite{Lem 1981}, \cite{Lem 1982}, \cite{Lem 1984}).

\begin{theorem} Let $D$ be a strongly linearly convex domain in $\mathbb C^n$, $n\geq 2$ or a convex domain in $\mathbb C^n$, $n\geq 1$. Then $c_D\equiv l_D$.
\end{theorem}

Any taut domain $D\subset\mathbb C^n$ for which the Lempert theorem holds is called a {\it Lempert domain}. The Lempert domain may be characterized in the following way. For any pair of different points $w,z\in D$ there is a holomorphic mapping $f:\mathbb D\to D$ passing through $w$ and $z$, and a holomorphic function $G:D\to\mathbb D$ such that $G(f(\lambda))=\lambda$, $\lambda\in\mathbb D$. In this case $f$ is called {\it the complex geodesic} and $G$ {\it its left inverse}. Sometimes the left inverses are understood up to the post-composition with an automorphism of $\mathbb D$, i. e. $(a\circ G)(f(\lambda))=\lambda$, $\lambda\in\mathbb D$, for some holomorphic automorphism $a$ of $\mathbb D$.
As a good reference for the theory around the Lempert theorem, to some extent also as reference on properties of holomorphically invariant functions that we use in the sequel, we recommend the monograph \cite{Jar-Pfl 2013}.

Denote the class of Lempert domains by $\mathcal L$. The Lempert Theorem states that all bounded convex and strongly linearly convex domains (as well as their biholomorphic images) belong to $\mathcal L$. Note that the equality of invariant functions as in the Lempert Theorem remains valid under taking the union of an increasing sequence of domains. Therefore, assuming the increasing sequence of Lempert domains is taut it will also be the Lempert domain.

In his proof Lempert showed more; namely, the left inverses he constructed were Lempert left inverses.
We call the left inverse $G$ {\it Lempert } (sometimes in the literature they are called {\it good}) if $G^{-1}(\lambda)$ is the intersection of an affine hyperplane with $D$. Note that the Lempert left inverse $G$ on $D$ is determined by the value of its derivative $G^\prime(f(\lambda))$: $G^{-1}(\lambda)=(f(\lambda)+\operatorname{ker}G^{\prime}(f(\lambda)))\cap D$, $\lambda\in\mathbb D$.

Denote by $\mathcal L_0$ a subclass of $\mathcal L$ of domains for which all complex geodesics have Lempert inverses. Therefore, assuming the union of increasing sequence of elements of $\mathcal L_0$ is taut it will also be an element of $\mathcal L_0$. As already mentioned all bounded convex and strongly linearly convex domains belong to the class $\mathcal L_0$.

In his proof Lempert showed the existence of duals that are closely related with left inverses. More precisely, let $D$ be a domain where the Lempert method works and let $f:\mathbb D\to D$ be a complex geodesic. 

The  holomorphic mapping $h:\mathbb D\mapsto \mathbb C^n$ is a {\it dual } to a complex geodesic $f$ if $h\in H^1(\mathbb D,\mathbb C^n)$ is not identically equal to $0$ and the following relation is satisfied: 
\begin{equation}\label{equation:dual}
    (z-f(G(z))\bullet h(G(z))=0,
\end{equation}
where $G$ is a left inverse to $f$.

Lempert proved that, under suitable assumptions on $D$, the above equation well defines the holomorphic left inverse $G$. That is the value $G(z)$, $z\in D$, is given as the only $\lambda\in\mathbb D$ such that
\begin{equation}\label{equation:left-inverse}
    (z-f(\lambda))\bullet h(\lambda)=0.
\end{equation}
The above discussion as well as an alternate approach to the proof of the Lempert theorem by Royden-Wong (\cite{Roy-Won 1983}) may be found not only in Lempert papers but also in \cite{Jar-Pfl 2013} and \cite{Aba 1989}.

In the paper we make a bit different approach to the existence of Lempert left inverses. We replace the duals with objects defined by the derivative of a left inverse and then we produce from the left inverse the Lempert left inverse and we present sufficient conditions for the uniqueness of left inverses (Theorem~\ref{theorem:lempert-left-inverses}). We then produce a family of Lempert left inverses connecting existing two Lempert left inverses to the same complex geodesic (Theorem~\ref{theorem:family-lempert-left-inverses}).

Further in Section~\ref{section:regularity-left-inverses} we present results on holomorphic extensions of left inverses through the boundary.

In Section~\ref{section:examples} we analyze special domains of the bidisc and the symmetrized bidisc where for some complex geodesics the structure of the class of Lempert left inverses is studied.

Additionally, we recall examples of left inverses in the ball showing the possible obstructions to the regularity of left inverses even in very regular domains.

\section{Left inverses induce Lempert left inverses}
As already mentioned we follow below the idea of Lempert but with a small modification. Namely, instead of considering the dual we shall work with some other function that coincides with the dual in regular situation. The attitude we present has advantage of being applicable in more situations and lets us conclude more properties of Lempert left inverses constructed. In many situations the functions coincide (up to some normalization constant) with the duals -- this is for instance the case of strongly convex domains. Unlike the duals the functions we examine behave well under the increasing sequence of domains.

We start with the following obvious observation. Having given a complex geodesic $f:\mathbb D\to D$ and a corresponding left inverse $G:D\to\mathbb D$ we get the equality
\begin{equation}
    G^{\prime}(f(\lambda))\bullet f^{\prime}(\lambda)=1,\; \lambda\in\mathbb D,
\end{equation}
where $w\bullet z:=\sum_{j=1}^nw_jz_j$, $w=(w_1,\ldots,w_n),\; z=(z_1,\ldots,z_n)\in\mathbb C^n$
It follows from a result of \cite{Aba 1989} (Lemma 2.6.44) that in regular situation ($D$ being strongly convex) the dual $h$ from the theory of Lempert is unique (up to a positive constant) and related with $G$ by the formula:
\begin{equation}
    h(\lambda)=G^{\prime}(f(\lambda)), \lambda\in\mathbb D.
\end{equation}

To avoid the necessity of taking the radial values of functions involved (which is a priori not guaranteed for the general case for the function $G^{\prime}\circ f$) we make below observation that lets us consider the Kobayashi balls rather than the whole domain and then we need not bother about the boundary regularity of the functions involved.

Recall that the sublevel sets of the Kobayashi distance with one argument fixed in convex domains behave well. More specifically, for the fixed $p\in D$, where $D$ is a bounded convex domain, the sublevel sets ({\it Kobayashi balls})
\begin{equation}
    B_{k_D^*}(p,R):=\{z\in D:k_D^*(p,z)<R\}
\end{equation}
are convex and 
\begin{equation}
    B_{k_D^*}(p,\rho R)=B_{k_{B_{k_D^*}(p,R)}^*}(p,\rho), \; 0<R<1, 0<\rho<1.
\end{equation}

As a consequence of the method of Royden-Wong (as presented in e. g. \cite{Jar-Pfl 2013}) we get the following.

\begin{theorem}\label{theorem:lempert-left-inverses} Let $D$ be a bounded convex domain. Let $f:\mathbb D\to D$ be a complex geodesic and let $G$ be its left inverse. Then for all $z\in D$ there is exactly one $\lambda=:H(z)\in\mathbb D$ such that
\begin{equation}
    (z-f(\lambda))\bullet G^{\prime}(f(\lambda))=0
\end{equation}
and $H:D\mapsto\mathbb D$ is holomorphic. 

Consequently, $H$ is the Lempert left inverse to $f$ with $H^{\prime}(f(\lambda))=G^{\prime}(f(\lambda))$, $\lambda\in\mathbb D$.

Moreover, if $G$ is a Lempert left inverse then $H\equiv G$. 

In the case the supporting hyperplanes are determined uniquely for the Kobayashi balls centered at $f(0)$ then there is only one Lempert left inverse to $f$ (this holds for instance in the case of strongly convex domains).
\end{theorem}
\begin{proof}
    Define the real hypersurface $M_R:=\{z\in D:|G(z)|^2=R^2\}$, $0<R<1$. Note that $M_R$ is disjoint from the convex domain $B_{k_D^*}(f(0),R)$ and for any $|\lambda_0|=R$ we have $f(\lambda_0)\in M_R\cap\partial(B_{k_D^*}(f(0),R))$, and $M_R$ is smooth (even real analytic) near $f(\partial\triangle(0,R))$. Consequently, the vector $\overline{G^{\prime}(f(\lambda_0))}G(f(\lambda_0))=\overline{G^{\prime}(f(\lambda_0))}\lambda_0$ is orthogonal to $M_R$ at $f(\lambda_0)$.  As $B_{k_D^*(f(0),R)}$ is a convex domain we easily get that the affine tangent space $f(\lambda_0)+T_{f(\lambda_0)}M_R$ to $M_R$ at $f(\lambda_0)$, $|\lambda_0|=R$, is disjoint from $B_{k_D^*(f(0),R)}$. Consequently, $f(\lambda_0)+T_{f(\lambda_0)}M_R$ is a supporting real affine hyperplane to $B_{k_D^*}(f(0),R)$ at $f(\lambda_0)$ and then
    \begin{equation}
        \operatorname{Re}\left((z-f(\lambda))\bullet \left(G^{\prime}(f(\lambda)))/\lambda\right)\right)<0, \; |\lambda|=R.
        \end{equation}
        Now we may use Lemma 11.2.2 from \cite{Jar-Pfl 2013} to get that there is the (Lempert) left inverse $H_R$ to $f_{|\triangle(0,R)}$ in $B_{k^*_D}(f(0),R)$. And thus passing with $R\to 1$ we easily get the Lempert left inverse $H$ to $f$ in $D$. 
    
    The equality of derivatives follows from the fact that the fiber of $H$ over $\lambda$ is determined by the kernel of $G^{\prime}(f(\lambda))$ and the uniformizing equality $H^{\prime}(f(\lambda)\bullet f^{\prime}(\lambda)=G^{\prime}(f(\lambda))\bullet f^{\prime}(\lambda)=1$.
    
The equality of $G$ and $H$ in the case $G$ is the Lempert left inverse follows directly from the equality of derivatives at $f(\lambda)$.

As the kernel of $H^{\prime}(f(\lambda))$ is a complex supporting hyperplane to $B_{k_D^*}(f(0),|\lambda|)$ at $f(\lambda)$ in the case the supporting hyperplanes are determined uniquely there is only one Lempert left inverse to a given complex geodesic.
\end{proof}

Now we show how we may produce a one parameter family of left inverses under the assumption that there exist two different Lempert left inverses.

\begin{remark} We start with the following obvious observation.
    For the given complex geodesic $f:\mathbb D\to D$ let $G_0, G_1:D\to\mathbb D$ be its different left inverses. Then certainly the functions (convex combinations of $G_0$ and $G_1$) $(1-t)G_0+tG_1$, $t\in[0,1]$, are (pairwise different) left inverses to $f$. However, in the case $G_0,G_1$ are Lempert left inverses their convex combination need not be a Lempert left inverse. But utilizing the previous theorem in place of the convex combination of original left inverses we may produce a one parameter family of Lempert left inverses joining the original Lempert left inverses. 
\end{remark}

\begin{theorem}\label{theorem:family-lempert-left-inverses} Let $D\subset\mathbb C^n$ be a bounded convex domain and let $f:\mathbb D\to D$ be a complex geodesic with two left inverses $G_0$ and $G_1$. Then there is a one parameter family of Lempert left inverses $H_t:D\to\mathbb D$, $t\in[0,1]$, for the complex geodesic $f$ connecting the corresponding left inverses $H_0$ and $H_1$.
\end{theorem}
\begin{proof} The functions $G_t:=(1-t)G_0+tG_1$ are left inverses for $f$. By Theorem~\ref{theorem:lempert-left-inverses} there are Lempert left inverses $H_t$ satisfying the equation:
\begin{equation}
    H_t^{\prime}(f(\lambda))=G_t^{\prime}(f(\lambda))=(1-t)G_0^{\prime}(f(\lambda))+tG_1^{\prime}(f(\lambda)),\;\lambda\in\mathbb D,
\end{equation}
which finishes the proof.
\end{proof}
\begin{remark} The above theorem is non-trivial even in the case $G_0$ and $G_1$ are different Lempert left inverses and it produces a one parameter family of Lempert left inverses connecting two given ones giving us a version of the convexity property in the class of Lempert left inverses. In other words any two Lempert left inverses for the given complex geodesic in the bounded convex domain are homotopic.
\end{remark}


\section{Regularity of (Lempert) left inverses}\label{section:regularity-left-inverses}

 

 Let us recall that the regularity properties of complex geodesics $f$, their duals $h$  and Lempert left inverses $G$ in strongly (linearly) convex domain $D$ may be found for instance in \cite{Lem 1981}, \cite{Lem 1982}, \cite{Lem 1984}, \cite{Aba 1989}). Let us list some of the properties of the functions mentioned.

 \begin{itemize}
     \item $f$  extends to a H\"older mapping on  $\overline{\mathbb D}$,\\
     \item the dual $h$ is uniquely determined (up to a multiple scalar) and it extends as a H\"older mapping on $\overline{\mathbb D}$ -- actually, up to a positive constant $h(\lambda)=G^{\prime}(f(\lambda))$, $\lambda\in\mathbb D$,\\
     \item $f^{\prime}(\lambda)\bullet h(\lambda)=C>0$, $\lambda\in\mathbb D$.\label{property:constant}
 \end{itemize}

Other results on regular extensions of complex geodesics may be found for instance in \cite{Mer 1993}, \cite{Cha-Hu-Lee 1988}, \cite{Bra-Pat-Tra 2009}, \cite{Hua-Wan 2022}, \cite{Bha 2016}, \cite{Zim 2022}. 

Let us prove now a result on holomorphic extension of the left inverse through a point of the boundary of a complex geodesic. Variations on the regular extensions of left inverses may be found in \cite{Aba 1989} (Theorem 2.6.43).

\begin{proposition} Let $D$ be a bounded convex domain in $\mathbb C^n$ and let $f$ be a complex geodesic in $D$, $h$ its dual and $G$ the Lempert left inverse to $f$ determined by $h$ (which in particular means that it satisfies the equation (\ref{equation:dual})). Assume further that $f^{\prime}\bullet h\equiv C>0$ on $\mathbb D$. Fix $z_0\in\partial D$ and let $\lambda_0\in\overline{\mathbb D}$ be such that $\lambda_0$ belongs to a cluster set $G^*(z_0)$. Assume that $f$ and $h$ extend holomorphically through $\lambda_0$ (which is trivially satisfied if $\lambda_0\in\mathbb D$). Then $G$ extends holomorphically through $z_0$.
\end{proposition}

\begin{proof}
Actually, consider the equation
\begin{equation}
    \Phi(z,\lambda):=(z-f(\lambda))\bullet h(\lambda)=0
\end{equation}
that for $z\in D$ and $\lambda\in D$ defines a Lempert left inverse $H$, $H(z):=\lambda$.
We know that $\Phi(f(\lambda_0),\lambda_0)=0$ and $\frac{\partial \Phi}{\partial\lambda}(f(\lambda_0),\lambda_0)=-f^{\prime}(\lambda_0)\bullet h(\lambda_0)$. The last expression is different from $0$.

Then the implicit mapping theorem finishes the proof.
\end{proof}
In the proof of the above proposition the crucial property that enabled us to use implicit function theorem was the fact that $f^{\prime}\bullet h\equiv C$ on $\mathbb D$. Note that replacing in the assumptions $h(\lambda)$ with $G^{\prime}(f(\lambda))$ the property will be satisfied so combining the above result (and its proof) with Theorem~\ref{theorem:lempert-left-inverses} we may formulate the following.

\begin{corollary}
    Let $D$ be a bounded convex domain in $\mathbb C^n$, let $f$ be a complex geodesic in $D$ and let $G$ be its left inverse. Assume that $f$ extends holomorphically through $\lambda_0\in\overline{\mathbb D}$ and $G$ extends holomorphically through $f(\lambda_0)\in\partial D$. Then the corresponding Lempert left inverse $H$ (the one existing by Theorem~\ref{theorem:lempert-left-inverses}) extends holomorphically through $f(\lambda_0)$. 
\end{corollary}
The above result means that the holomorphic extension property of the left inverse implies the holomorphic extension property of the corresponding Lempert left inverse.

We close this subsection with some further properties. First we have

\begin{corollary} Let $D$ be a strongly convex domain in $\mathbb C^n$. Assume that the complezx geodesic $f$ and its dual extend holomorphically through the boundary. Then the corresponding Lempert left inverse extends holomorphically through the boundary of $D$.
\end{corollary}

And an immediate consequence of a result of Lempert we get the following known property (see \cite{Lem 1981}, \cite{Aba 1989}, compare also \cite{Kos-War 2013}).

\begin{corollary} Let $D$ be a strongly convex domain with real analytic boundary. Then the Lempert left inverses extend holomorphically through $\overline{D}$. Moreover, for $z\in\overline{D}$ $|G(z)|=1$ iff $z\in f(\partial\mathbb D)$.
\end{corollary}

\subsection{Obstruction to regular extensions of left inverses}

It is known that for very regular domains there exist left inverses that are not Lempert which shows that the uniqueness of left inverses to complex geodesics does not hold generally (for results on that topic see for instance \cite{Kos-Zwo 2016}). The simplest example in that context that also shows the noncontinuity of left inverses (being certainly non Lempert) in very regular situation is the function
\begin{equation}
    \mathbb B_2\owns z\to \frac{z_1}{\sqrt{1-z_2^2}}\in\mathbb D
\end{equation}
that is a left inverse to the complex geodesic $\mathbb D\owns \lambda\to (\lambda,0)\in\mathbb B_2$ and it does not extend continuously through the points $(0,\pm 1)$.

A more refined example of a (non-Lempert) left inverse in the unit ball $\mathbb B_2$ is the following example (see \cite{Kos-Zwo 2018}, compare also \cite{Kos-Zwo 2021})
\begin{equation}
G:\mathbb B_2\owns z\to
\frac{2z_1(1-z_1)-z_2^2}{2(1-z_1)-z_2^2}\in\mathbb D
\end{equation}
that is a left inverse (up to a post-composition with an automorphism of the unit disc) to a 'big' family of complex geodesics
\begin{equation}
    f_t(\lambda)=
    \left(\frac{t^2+\lambda}{1+t^2},\frac{t(\lambda-1)}{1+t^2}\right),\;\lambda\in\mathbb D,t\in\mathbb R.
\end{equation}
The function $G$ does not extend continuously through the only point $(1,0)$ and extends holomorphically through all other boundary points. It also assumes at the boundary of $\mathbb B_2$ the values from $\partial\mathbb D$ for all $z\in\partial\mathbb B_2\setminus\{(1,0)\}$ iff $\operatorname{Im}(z_2(1-\overline{z_1}))=0$, which shows that the behaviour of left inverses in very regular domain (the ball) differs significantly from the behaviour of the Lempert left inverse.


\section{Examples of domains admitting many Lempert left inverses}\label{section:examples}
As we saw in regular domains (strongly convex) the Lempert left inverses are uniquely determined (unlike the case of general left inverses). We present below the situation in a less regular domains and we concentrate on two important classical examples.

\subsection{The bidisc}
In the case of the bidisc the study of complex geodesics may be reduced (up to holomorphic automorphisms of the bidisc) to the two cases. We shall consider them below.

First consider the complex geodesic
\begin{equation}
    \mathbb D\owns\lambda\to(\lambda,\lambda\psi(\lambda)),\lambda\in\mathbb D,
\end{equation}
where $\psi:\mathbb D\to\mathbb D$ is holomorphic.

Then the (only) left inverse (and thus the only Lempert left inverse) is the projection on the first coordinate. 


On the other hand the complex geodesic
\begin{equation}
    f_d:\mathbb D\owns\lambda\to(\lambda,\lambda)\in\mathbb D^2
\end{equation}
has many left inverses. More precisely, any left inverse to $f_d$  hast the following form
(see \cite{Agl-McC 2002}, Example 11.79):
\begin{equation}\label{equation:left-inverse-bidisc}
    G(z)=\frac{tz_1+(1-t)z_2-z_1z_2h(z)}{1-((1-t)z_1+tz_2)h(z)}, 
\end{equation}
where $t\in[0,1]$ and $h:\mathbb D^2\to\overline{\mathbb D}$ is holomorphic.

There are also many Lempert left inverses to $f_d$. For instance, for $t\in[0,1]$ the function
$\mathbb D^2\owns z\to tz_1+(1-t)z_2\in\mathbb D$ is a Lempert left inverse to $f_d$.

Below we shall prove that these are all Lempert left inverses to $f_d$.

\begin{proposition}
    Any Lempert left inverse to the complex geodesic $f_d$ is of the following form: $G(z)=tz_1+(1-t)z_2$, $t\in[0,1]$, $z\in\mathbb D^2$.
\end{proposition}
\begin{proof} Let $G$ be a Lempert inverse. It must be as in (\ref{equation:left-inverse-bidisc}).
    In the case $t=0$ or $t=1$ we get that $G$ is one of the projections.

Assume then that $t\in(0,1)$ and let $G$ be a Lempert left inverse. Then $G$ must be as above such that for any $\alpha\in\mathbb D$ we get $v\in\mathbb C^2\setminus\{(0,0)\}$ such that
\begin{equation}
    G(\alpha(1,1)+\lambda v)=\alpha
\end{equation}
for all $\lambda$ with $|\alpha+\lambda v_1|,|\alpha+\lambda v_2|<1$. Transforming the last formula we get that
\begin{equation}
    \lambda(t v_1+(1-t)v_2)=h(\alpha(1,1)+\lambda v)\lambda(\alpha(tv_1+(1-t)v_2)+\lambda v_1v_2).
\end{equation}
This gives for $\lambda \neq 0$
\begin{equation}
(\alpha(tv_1+(1-t)v_2)+\lambda v_1v_2)h(\alpha(1,1)+\lambda v)=tv_1+(1-t)v_2.
\end{equation}
Consequently, either $h(\alpha(1,1))=1/\alpha$, which give a contradiction or $tv_1+(1-t)v_2=0$. In the latter case we may choose $v=(1-t,-t)$ and then $h(\alpha(1,1)+\lambda(1-t,-t))=0$ for all $\alpha\in\mathbb D$ and sufficiently small $\lambda\in\mathbb D$. Then by the identity principle $h\equiv 0$ which finishes the proof.
\end{proof}
\begin{remark}
    Note that unlike in the case of the unit ball the holomorphic automorphsims of the bidisc do not transform affine complex lines to affine complex lines so the fact that we gave a complete description of Lempert left inverses for $f_d$ does not mean that we get a complete description of Lempert left inverses of complex geodesics of the form
    \begin{equation}
        \mathbb D\owns\lambda\to(\lambda,a(\lambda))\in\mathbb D,
    \end{equation}
    where $a$ is an automorphism of $\mathbb D$.
\end{remark}

\subsection{The symmetrized bidisc}
The symmetrized bidisc $\mathbb G_2$ defined as
\begin{equation}
    \mathbb G_2:=\{(\lambda_1+\lambda_2,\lambda_1\lambda_2):\lambda_1,\lambda_2\in\mathbb D\}
\end{equation}
is a Lempert domain that is not biholomorphic to a convex domain or even cannot be exhausted by domains that are biholomorphic to convex domains (see \cite{Agl-You 2004}, \cite{Cos 2004}, \cite{Edi 2004}). However, it is an increasing union of strongly linearly convex domains (see \cite{Pfl-Zwo 2012}) which consequently by the standard properties of invariant functions implies directly from the papers of Lempert that it is a Lempert domain and even more, it is from the class $\mathcal L_0$.


There is a complete description of all complex geodesics in the symmetrized bidisc. They are uniquely determined and extend holomorphically through the boundary of $\mathbb D$ which differs from the case of the bidisc (see \cite{Agl-You 2006}, \cite{Pfl-Zwo 2005}).

Unlike in the case of the bidisc any holomorphic automorphism of the symmetrized bidisc maps intersections of $\mathbb G_2$ with affine lines to intersections of $\mathbb G_2$ with affine lines. 

We know a family of Lempert left inverses from which we may find left inverses to all complex geodesics in the symmetrized bidisc. Define
\begin{equation}
  \Psi_{\omega}(s,p):=\frac{2p-\omega s}{2-\overline{\omega} s}, \; (s,p)\in\mathbb G_2, \; |\omega|\leq 1.
\end{equation}
The family $\{\Psi_{\omega}:|\omega|=1\}$ may replace the family of all bounded holomorphic functions in the definition of the Carath\'eodory distance and thus for any complex geodesic in the symmetrized bidisc there is a left inverse of the form $\Psi_{\omega}$, $|\omega|=1$ (up to a composition with an automorphism of the unit disc).

Note that $\Psi_{\omega}$, $|\omega|=1$, does not extend continuously through the point $(2\omega,\omega^2)$. On the other hand $\Psi_{\omega}$, $\omega\in\mathbb D$, extends holomorphically through all other boundary points of $\mathbb G_2$.

In \cite{Kos-Zwo 2016} the authors found a complete characterization of uniqueness of left inverse to complex geodesics. In most cases the complex geodesics admit only one left inverse (that is then the Lempert left inverse). The cases of complex geodesics with non-uniqueness comprise the royal geodesic and flat geodesics that we discuss below. And a very special case of complex geodesics omitting the royal variety that we shall not discuss.

\subsection{Lempert left inverse to the royal geodesic}
We start with the {\it royal geodesic} that is the geodesic of the form
\begin{equation}
    f_r:\mathbb D\owns\lambda\to (2\lambda,\lambda^2)\in\mathbb G_2.
\end{equation}

We have a complete description of all left inverses to $f_r$.
\begin{proposition}\label{proposition:left-inverses-royal-geodesics}
    The mapping $\Psi:\mathbb G_2\to\mathbb D$ is the left inverse to the complex geodesic $f_r$ iff 
    \begin{equation}
      \Psi(s,p)=\frac{s-2p h(s,p)}{2-s h(s,p)},\; (s,p)\in\mathbb G_2,
\end{equation}
where $h\in\mathcal O(\mathbb G_2,\overline{\mathbb D})$.  
    \end{proposition}
\begin{proof}
    It follows from the description of left inverses to complex geodesics in the bidisc that the arbitrary left inverse $\Psi$ to $f_r$  is such that $\Phi:=\Psi\circ \pi$ , is the left inverse to the complex geodesic $f_d$ in $\mathbb D^2$ so it is of the form as in (\ref{equation:left-inverse-bidisc}) with the additional symmetry condition. It means, in particular, that
\begin{equation}
    t=\frac{\partial\Phi(0,0)}{\partial z_1}=\frac{\partial \Phi(0,0)}{\partial z_2}=1-t,
\end{equation}
which means in particular that $t=1/2$ and easily finishes the proof.
\end{proof}

 Note that the functions $\Psi_{\omega}$, $|\omega|\leq 1$ being the left inverses (up to a composition with the disc automorphisms) of complex geodesics of the symmetrized bidisc are Lempert left inverses. 

Below we present an implicit way of the application of Theorem~\ref{theorem:lempert-left-inverses} and provide more examples of Lempert left inverse to the royal geodesic. Note that we cannot apply this theorem as the symmetrized bidisc is not convex (nor even biholomorphic to a convex domain) but we show how the construction presented in Theorem~\ref{theorem:lempert-left-inverses} leads to the formulas for new (in our opinion unexpected) Lempert left inverses to $f_r$. But as
 we cannot apply Theorem~\ref{theorem:lempert-left-inverses} we have to, after establishing the possible formula for the Lempert left inverse, verify that the function given by that formula is in fact the left inverse for $f_r$.

Note that the functions $-\overline{\omega}\Psi_{\omega}$, $|\omega|=1$, are left inverses to $f_r$ and
\begin{equation}
    (-\overline{\omega}\Psi_{\omega})^{\prime}(s,p)=\frac{2}{(2-\overline{\omega}s)^2}\left(\overline{\omega}(\omega-p\overline{\omega}), \overline{\omega}^2s-2\overline{\omega}\right),\; (s,p)\in\mathbb G_2.
\end{equation}
The equation in Theorem~\ref{theorem:lempert-left-inverses} describing the Lempert left inverse would be
\begin{equation}
    \left((2\lambda,\lambda^2)-(s,p)\right)\bullet \frac{2}{(2-\overline{\omega}2\lambda)^2}\left(1-\overline{\omega}^2\lambda^2,-\overline{\omega}(2-\overline{\omega}2\lambda)\right)=0.
\end{equation}
Now  the way the intermediate Lempert left inverses are produced in Theorem~\ref{theorem:family-lempert-left-inverses} gives us (in the case the left iverses related with $\omega_2=-\omega_1=-1$ and $t=1/2$ are considered) the following relation
\begin{equation}
    -s\lambda^2+2\lambda(1+p)-s=0.
\end{equation}
Consequently, our candidate for the Lempert left inverse would be the function
\begin{equation}
    \Phi(s,p):=\frac{s}{1+p+\sqrt{(1+p)^2-s^2}},\;(s,p)\in\mathbb G_2,
\end{equation}
where the square root in the definition of $\Phi$ is chosen so that $\sqrt{1}=1$. At this place note that for $(s,p)=(\lambda_1+\lambda_2,\lambda_1\lambda_2)\in\mathbb G_2$ we have $(1-p)^2+s^2=(1-\lambda_1^2)(1-\lambda_2^2)$ so the square root in the definition of $\Phi$ (and thus the function $\Phi$) is a well-defined holomorphic function.


What remains to prove is to show that $\left|\Phi(s,p)\right|<1$, $(s,p)\in\mathbb G_2$.

Note that $\Phi(s,p)=\frac{s/(1+p)}{1+\sqrt{1-\left(s/(1+p)\right)^2}}$.
In other words we have the composition of two functions $\Phi=\alpha\circ\beta$,
$\beta(s,p):=\frac{s}{1+p}$, $\alpha(\lambda):=\frac{\lambda}{1+\sqrt{1-\lambda^2}}$. The calculations carried out on p. 153 in \cite{Nik-Pfl-Zwo 2008} give the following
\begin{equation}
    \beta(\mathbb G_2)=\bigcup_{\lambda\in\mathbb D}\triangle\left(\frac{2i\operatorname{Im}\lambda}{1-|\lambda|^2},\frac{|1-\lambda^2|}{1-|\lambda|^2}\right).
\end{equation}
As $\left|\pm 1+\frac{2i\operatorname{Im}\lambda}{1-|\lambda|^2}\right|=\frac{|1-\lambda^2|}{1-|\lambda|^2}$ we get that $\beta(\mathbb G_2)=\mathbb C\setminus((-\infty,-1]\cup [1,\infty))=:A$. On the other hand one may verify that $\alpha(A)=\mathbb D$ which completes the proof of our claim.

\begin{remark} It would be desirable to try to find all the Lempert inverses to the royal geodesic. As we saw in Proposition~\ref{proposition:left-inverses-royal-geodesics} the full description of left inverses to the royal geodesic is known. But it seems that a nice description of all the Lempert left inverses could be more challenging.
\end{remark}

\subsection{Analysis of the case of flat geodesics} Another class of complex geodesics admitting many (Lempert) left inverses are so called {\it flat geodesics}. They are of the form
\begin{equation}
    \mathbb D\owns\lambda\to(\beta+\overline{\beta}\lambda,\lambda)\in\mathbb G_2,
\end{equation}
where $|\beta|<1$. As all flat geodesics are biholomorphically equivalent we may assume that $\beta=0$. All the functions $\Psi_{\omega}$, $|\omega|\leq 1$ are its Lempert left inverses (up to rotations). Let us repeat the reasoning as in the case of the royal geodesic. The formula for the left inverse in Theorem~\ref{theorem:lempert-left-inverses} would lead to the following equation
\begin{equation}
    ((0,\lambda)-(s,p))\bullet\left(\lambda\overline{\omega}-\omega,2\right)=0.
\end{equation}
The reasoning as in the case of the royal geodesic that imitates the procedure introduced in Theorem~\ref{theorem:family-lempert-left-inverses} with $\omega_1,\omega_2\in\overline{\mathbb D}$, $t\in[0,1]$ leads to the solution of $\lambda=\Psi_{t\omega_1+(1-t)\omega_2}(s,p)$. And thus in this case the procedure delivers the known Lempert left inverses.


\begin{remark}
    The above considerations leave a lot of problems open; some of them have already been mentioned. As to other problems one may ask the question whether the family of all Lempert left inverses to a given complex geodesic in the symmetrized bidisc is always connected. Can one have a counterpart of Theorem~\ref{theorem:lempert-left-inverses} in a wider class of domains or applied to special complex geodesics, perhaps for the symmetrized bidisc? Can one find all (Lempert) left inverses to flat geodesics or a royal geodesic?
\end{remark}

\end{document}